\documentclass[12 pt]{article}
\usepackage{stmaryrd}
\usepackage{times}
\usepackage{booktabs}
\usepackage{subfigure}
\usepackage{pifont}
\usepackage{floatrow}
\usepackage{caption}
\usepackage{mathrsfs}
\usepackage[fleqn]{amsmath}
\usepackage{amsfonts,amsthm,amssymb,mathrsfs,bbding}
\usepackage{txfonts}
\usepackage{graphics,multicol}
\usepackage{graphicx}
\usepackage{color}
\usepackage{caption}
\usepackage{indentfirst}
\usepackage{cite}
\usepackage{latexsym,bm}
\usepackage{enumerate}
\evensidemargin=\oddsidemargin\topmargin=-1.5cm

\newtheorem{lem}{Lemma}[section]

\newtheorem{thm}{Theorem}[section]
\newtheorem{cor}{Corollary}[section]

\newtheorem{remark}{Remark}

\newtheorem{exa}{Example}
\theoremstyle{definition}

\addtocounter{section}{0}

\begin{document}
\title{Integral Cayley Graphs over Dihedral Groups\footnote{Supported
by the National Natural Science Foundation of China (Grant Nos. 11261059, 11531011).}}
\author{{\small Lu Lu, \ \ Qiongxiang Huang\footnote{
Corresponding author.
Email: huangqx@xju.edu.cn},\ \ Xueyi Huang}\\[2mm]\scriptsize
College of Mathematics and Systems Science,
\scriptsize Xinjiang University, Urumqi, Xinjiang 830046, P.R. China}

\date{}
\maketitle {\flushleft\large\bf Abstract }
In this paper, we give a necessary and sufficient condition for the integrality of Cayley graphs over the dihedral group $D_n=\langle a,b\mid a^n=b^2=1,bab=a^{-1}\rangle$. Moreover, we also obtain some simple sufficient conditions for the integrality of Cayley graphs over $D_n$ in terms of the Boolean algebra of $\langle a\rangle$, from which we find infinite classes of integral Cayley graphs over $D_n$. In particular, we completely determine all integral Cayley graphs over the dihedral group $D_p$ for a prime $p$.
\vspace{0.1cm}
\begin{flushleft}
\end{flushleft}
\section{Introduction}\label{s-1}
A graph $X$ is said to be \emph{integral} if all eigenvalues of the adjacency matrix of $X$ are integers. The property was first defined by Harary and Schwenk \cite{Harary}, who suggested the problem of classifying integral graphs. This problem initiated a significant investigation among algebraic graph theorists, trying to construct and classify integral graphs. Although this problem is easy to state, it turns out to be extremely hard. It has been attacked by many mathematicians during the past 40 years and it is still wide open.

Since the general problem of classifying integral graphs seems too difficult, graph theorists started to investigate special classes of graphs, including trees, graphs with bounded degrees, regular graphs, and Cayley graphs. The first considerable result on integral trees was given by Watanabe and Schwenk in \cite{Watanabe-1} and \cite{Watanabe-2}. Then many mathematicians constructed some infinite classes of integral trees with bounded diameters (\cite{Cheng,Wang,Brouwer}). It is heartening that Csikv\'{a}ri \cite{Csikvari} constructed integral trees with arbitrary large diameters. With regard to integral regular graphs, the first significant result was given by Bussemaker and Cvetkovi\'{c} \cite{Bussemaker} in 1976, which showed that there are only 13 connected cubic integral graphs. About 20 years later, in 2000, Stevanovi\'{c} considered the 4-regular integral graphs avoiding $\pm3$ in the spectrum and gave the possible spectra of 4-regular bipartite integral graphs without $\pm3$ as their eigenvalues in \cite{Stevanovic}. Recently, Lepovi\'{c} \cite{Lepovic} proposed that there exist exactly $93$ non-regular, bipartite integral graphs with maximum degree $4$.

Given a finite group $G$ and a subset $1\not\in S\subseteq G$ with $S=S^{-1}$, the Cayley graph $X(G,S)$ has vertex set $G$ and two vertices $a,b$ are adjacent if $a^{-1}b\in S$. In 2005, So \cite{So} gave a complete characterization of integral circulant graphs. Later, Abdollahi and Vatandoost \cite{Abdollahi} showed that there are exactly seven connected cubic integral Cayley graphs in 2009. About the same year, Klotz and Sander \cite{Klotz} proved that, for an abelian group $G$, if the Cayley graph $X(G,S)$ is integral then $S$ belongs to the Boolean algebra $B(\mathcal{F}_G)$ generalized by the subgroups of $G$. Moreover, they conjectured that the converse is also true, which has been proved by Alperin and Peterson \cite{Alperin}.

In 2014, Cheng, Terry and Wong (cf. \cite{Cheng}, Corollary 1.2) presented that the normal Cayley graphs over symmetric groups are integral (a Cayley graph is said to be \emph{normal} if its generating set $S$ is closed under conjugation). It seems that there are few works about the characterization of integral Cayley graphs over non-abelian groups. As a simple attempt to this aspect, we try to characterize integral Cayley graphs over dihedral groups. At first, by using the expression of spectra of Cayley graphs, we obtain the necessary and sufficient conditions for the integrality of Cayley graphs over the dihedral group $D_n$ (see Theorems \ref{thm-3-1} and \ref{thm-3-2}). In terms of atoms of Boolean algebra of $D_n$ we also obtain a simple sufficient condition (see Corollary \ref{cor-3-3}) and a necessary condition  (see Corollary \ref{cor-3-4}) for the integrality of Cayley graphs over $D_n$. In particular, we determine all integral Cayley graphs over $D_p$ for a prime $p$ (see Theorem \ref{cor-4-1}).

\section{The spectra of Cayley graphs over dihedral groups}\label{s-2}
First of all, we review some basic definitions and notions of representation theory for latter use.

Let $G$ be a finite group and $V$ a $n$-dimensional vector space over $\mathbb{C}$. A \emph{representation} of $G$ on $V$ is a group homomorphism $\rho:G\rightarrow GL(V)$, where $GL(V)$ denotes the group of automorphisms of $V$. The \emph{degree} of $\rho$ is the dimension of $V$. Two representations $\rho_1:G\rightarrow GL(V_1)$ and $\rho_2:G\rightarrow GL(V_2)$  of $G$ are called \emph{equivalent}, written as $\rho_1\sim\rho_2$, if there exists an isomorphism $T: V_1\rightarrow V_2$ such that $T\rho_1(g)=\rho_2(g)T$ for all $g\in G$.

Let $\rho:G\rightarrow GL(V)$ be a representation. The \emph{character} $\chi_{\rho}:G\rightarrow \mathbb{C}$ of $\rho$ is defined by setting $\chi_{\rho}(g)=Tr(\rho(g))$ for $g\in G$, where $Tr(\rho(g))$ is the trace of the representation matrix of $\rho(g)$ with respect to some basis of $V$. The \emph{degree} of the character $\chi_{\rho}$ is just the degree of $\rho$, which equals to $\chi_{\rho}(1)$. A subspace $W$ of $V$ is said to be
\emph{$G$-invariant} if $\rho(g)w\in W$ for each $g\in G$ and $w\in W$. If $W$ is a $G$-invariant subspace of $V$, then the restriction of $\rho$ on $W$, i.e., $\rho_{|W}:G\rightarrow GL(W)$, is a representation
of $G$ on $W$. Obviously, $\{1\}$ and $V$ are always $G$-invariant subspaces, which are called \emph{trivial}. We say that $\rho$ is an \emph{irreducible representation} and $\chi_{\rho}$ an \emph{irreducible character} of $G$, if $V$ has no nontrivial $G$-invariant subspaces. One can refer to \cite{Jean} for more information about representation theory.

If we  build synthetically a vector space $\mathbb{C}G$ whose basis consists of  the elements of $G$, i.e.,
$$\mathbb{C}G=\{\sum_{g\in G}c_gg\mid c_g\in \mathbb{C}\},$$
then the \emph{(left) regular representation} of $G$ is the homomorphism $L:G\rightarrow GL(\mathbb{C}G)$ defined by
$$L(g)\sum_{h\in G}c_hh=\sum_{h\in G}c_hgh=\sum_{x\in G}c_{g^{-1}x}x$$
for each $g\in G$. The following result about regular representation is well known.
\begin{lem}[\cite{Jean}]\label{lem-2-1}
Let $L$ be the regular representation of $G$. Then
$$L\sim d_1\rho_1\oplus d_2\rho_2\oplus\cdots\oplus d_h\rho_h,$$
where $\rho_1,\ldots,\rho_h$ are all non-equivalent irreducible representations of $G$ and $d_i$ is the degree of $\rho_i$ ($1\le i\le h$).
\end{lem}

Suppose that $X=X(G,S)$ is an undirected Cayley graph without loops, that is, $S$  is inverse closed and does not contain the identity.  Let $L$ be the regular representation of $G$, and $R(g)$ the representation matrix corresponding to $L(g)$ for $g\in G$. Babai \cite{Babai} expressed the adjacency matrix of $X(G,S)$ in terms of $R(g)$.
\begin{lem}[\cite{Babai}]\label{lem-2-2}
Let $G$ be a finite group of order $n$, and let $S\subseteq G\setminus \{1\}$ be such that $S=S^{-1}$. Then the adjacency matrix $A$ of $X(G,S)$ can be expressed as $A=\sum_{s\in S}R(s)$, where $R(s)$ is the representation matrix corresponding to $L(s)$.
\end{lem}
Let $\rho_1,\ldots,\rho_h$ be all non-equivalent irreducible representations of $G$ with degrees $d_1,\ldots,d_h$ ($d_1^2+\cdots+d_h^2=n$), respectively. Denote by $R_i(g)$ the representation matrix corresponding to $\rho_i(g)$ for $g\in G$ and $1\le i\le h$. Therefore, by Lemma \ref{lem-2-1},
$$L\sim d_1\rho_1\oplus\cdots\oplus d_h\rho_h,$$
which means that there exists an orthogonal matrix $P$ such that
$$PR(g)P^{-1}=d_1R_1(g)\oplus d_2R_2(g)\oplus\cdots\oplus d_hR_h(g)$$
for each $g\in G$. Therefore, we have
$$PAP^{-1}=P(\sum_{s\in S}R(s))P^{-1}=d_1\sum_{s\in S}R_1(s)\oplus\cdots \oplus d_h\sum_{s\in S}R_h(s).$$
Suppose that $\lambda_{i,1},\ldots,\lambda_{i,d_i}$ are all  eigenvalues of the matrix $\sum_{s\in S}R_i(s)$ ($1\le i\le h$), the spectrum of $X(G,S)$ is given by
$$\textrm{Spec}(X(G,S))=\{[\lambda_{1,1}]^{d_1},\ldots,[\lambda_{1,d_1}]^{d_1},\ldots,[\lambda_{h,1}]^{d_h},\ldots,[\lambda_{h,d_h}]^{d_h}\}.$$
Thus we have
\begin{lem}[\cite{Babai}]\label{lem-2-3}
The spectrum of  $X(G,S)$ is given by
$$\textrm{Spec}(X(G,S))=\{[\lambda_{1,1}]^{d_1},\ldots,[\lambda_{1,d_1}]^{d_1},\ldots,[\lambda_{h,1}]^{d_h},\ldots,[\lambda_{h,d_h}]^{d_h}\},$$
where
$\lambda_{i,1}^t+\lambda_{i,2}^t+\cdots+\lambda_{i,d_i}^t=\sum_{s_1,\ldots,s_t\in S}\chi_{\rho_i}(\prod_{k=1}^ts_t)$
for any natural number $t$.
\end{lem}
Denote by $D_n=\langle a,b\mid a^n=b^2=1,bab=a^{-1}\rangle$ the dihedral group of order $2n$. Now we  list the character table of $D_n$.
\begin{lem}[\cite{Jean}]\label{lem-2-4}
The character table of  $D_n$ is given in Table 1 if $n$ is odd, and  in Table 2 otherwise, where $\psi_i$ and $\chi_j$ are irreducible characters of degree one and two, respectively, and $1\le h\le [\frac{n-1}{2}]$.
\renewcommand{\arraystretch}{1}
\begin{table}[ht]
\floatsetup{floatrowsep=quad,captionskip=7pt} \tabcolsep=20pt
\begin{center}{\footnotesize
\begin{floatrow}
\ttabbox{\caption{\label{tab-1}\footnotesize{Character table of $D_n$ for odd $n$}}}{
\begin{tabular}{ccc}
\toprule
      &$a^k$& $ba^k$ \\
  \midrule
  $\psi_{1}$&$1$& $1$\\
  $\psi_{2}$&$1$& $-1$\\
  $\chi_{_h}$&$2\cos(\frac{2kh\pi}{n})$& $0$\\
  \bottomrule
\end{tabular}}
\ttabbox{\caption{\label{tab-2}\footnotesize{Character table of $D_n$ for even $n$}}}{
\begin{tabular}{ccc}
  \toprule
     &$a^k$& $ba^k$ \\
  \midrule
  $\psi_{1}$&$1$& $1$\\
  $\psi_{2}$&$1$& $-1$\\
  $\psi_{3}$&$(-1)^k$& $(-1)^k$\\
  $\psi_{4}$&$(-1)^k$& $(-1)^{k+1}$\\
  $\chi_{_h}$&$2\cos(\frac{2kh\pi}{n})$& 0\\
  \bottomrule
 \end{tabular}}
\end{floatrow}}
\end{center}
\end{table}
\end{lem}
By Lemmas \ref{lem-2-3} and \ref{lem-2-4}, we get the spectra of Cayley graphs over $D_n$ immediately.
\begin{thm}\label{thm-2-1}
Let $D_n$ be a dihedral group and $S\subseteq D_n\setminus\{1\}$ satisfying $S=S^{-1}$. Then
$$\textrm{Spec}(X(D_n,S))=\{[\lambda_i]^1;[\mu_{h1}]^2, [\mu_{h2}]^2\mid i=1,\ldots,m;h=1,2,\ldots,[\frac{n-1}{2}]\},$$
where $m=2$ if $n$ is odd and $m=4$ otherwise, and
\begin{equation}\label{spec-Dn-1}
\left\{\begin{array}{ll}\lambda_i=\sum_{s\in S}\psi_{i}(s),& \mbox{$i=1,\ldots,m$;}\\
\mu_{h1}+\mu_{h2}=\sum_{s\in S}\chi_{_h}(s),& h=1,2,\ldots,[\frac{n-1}{2}];\\
\mu_{h1}^2+\mu_{h2}^2=\sum_{s_1,s_2\in S}\chi_{_h}(s_1s_2),& h=1,2,\ldots,[\frac{n-1}{2}].
\end{array}\right.
\end{equation}
\end{thm}

Let $A$, $B$ be two subsets of a group $G$.  For any character $\chi$ of $G$, we denote $\chi(A)=\sum_{a\in A}\chi(a)$ and $\chi(AB)=\sum_{a\in A,b\in B}\chi(ab)$. Particularly, $\chi(A^2)=\sum_{a_1,a_2\in A}\chi(a_1a_2)$.
\begin{thm}\label{thm-3-1}
Let $D_n=\langle a,b\mid a^n=b^2=1,bab=a^{-1}\rangle$ be the dihedral group, and let $S=S_1\cup S_2\subseteq D_n\setminus\{1\}$ be such that $S=S^{-1}$, where $S_1\subseteq\langle a\rangle$ and $S_2\subseteq b\langle a\rangle$. Then $X(D_n,S)$ is integral if and only if the following two conditions hold for $1\le h\le [\frac{n-1}{2}]$:
\vspace{-0.2cm}
\begin{enumerate}[(i)]
\item $\chi_{_h}(S_1)$, $\chi_{_h}(S_1^2)+\chi_{_h}(S_2^2)$ are integers;
\vspace{-0.2cm}
\item $\Delta_h(S)=2[\chi_{_h}(S_1^2)+\chi_{_h}(S_2^2)]-[\chi_{_h}(S_1)]^2$ is a square number.
    \vspace{-0.2cm}
\end{enumerate}
\end{thm}
\begin{proof}
 Note that $S_1S_2=\{s_1s_2\mid  s_1\in S_1, s_2\in S_2\}\subseteq b\langle a\rangle$ and $S_2S_1\subseteq b\langle a\rangle$. By Lemma \ref{lem-2-4}, we have $\chi_h(S_1S_2)=0=\chi_h(S_2S_1)$. Thus
$$\begin{array}{lll}
\chi_{_h}(S)&=&\sum_{s_1\in S_1}\chi_{_h}(s)+\sum_{s_2\in S_2}\chi_{_h}(s_2)=\chi_h(S_1),\\
\chi_{_h}(S^2)&=&\sum_{s_1,s_2\in S}\chi_{_h}(s_1s_2)\\
&=&\chi_{_h}(S_1^2)+\chi_{_h}(S_1S_2)+\chi_{_h}(S_2S_1)+\chi_{_h}(S_2^2)\\
&=&\chi_{_h}(S_1^2)+\chi_{_h}(S_2^2).
\end{array}
$$
Then the spectrum of $X(D_n,S)$ presented in (\ref{spec-Dn-1}) can be rewritten as
\begin{equation}\label{spec-Dn-2}
\left\{\begin{array}{ll}\lambda_i=\sum_{s\in S}\psi_{i}(s)=\psi_i(S),& i=1,\ldots,m;\\
\mu_{h1}+\mu_{h2}=\chi_{_h}(S_1),&  h=1,2,\ldots,[\frac{n-1}{2}];\\
\mu_{h1}^2+\mu_{h2}^2=\chi_{_h}(S_1^2)+\chi_{_h}(S_2^2),&  h=1,2,\ldots,[\frac{n-1}{2}].\\
\end{array}\right.\\
\end{equation}

First suppose that $X(D_n,S)$ is integral.  From  (\ref{spec-Dn-2}), we know that $\chi_{_h}(S_1)$ and $\chi_{_h}(S_1^2)+\chi_{h}(S_2^2)$ must be integers, and thus (i) holds. Since $\mu_{h1}$ and  $\mu_{h2}$ are integers, and they are also the roots of the following quadratic equation:
\begin{equation}\label{spec-Dn-3}
x^2-\chi_{_h}(S_1)\cdot x+\frac{1}{2}[\chi_{_h}(S_1)^2-(\chi_{_h}(S_1^2)+\chi_{_h}(S_2^2)]=0,
\end{equation}
the discriminant $\Delta_h(S)=2[\chi_{_h}(S_1^2)+\chi_{_h}(S_2^2)]-[\chi_{_h}(S_1)]^2$ must be a square number, and thus (ii) follows.

Next suppose that (i) and (ii) hold. Then the solutions $\mu_{h1}$ and  $\mu_{h2}$ of  (\ref{spec-Dn-3}) must be rational. This implies that $\mu_{h1}$ and  $\mu_{h2}$ must be integers because they are algebraic integers. Additionally, the eigenvalues $\lambda_i$ are always integers. Hence $X(D_n,S)$ is integral.
\end{proof}
Let $C_n=\langle a\rangle$ be the cyclic group of order $n$. It is well known in \cite{Jean} that the irreducible characters of $C_n$ can be presented by
\begin{equation}\label{phi-1}\phi_h(a^k)=e^{\frac{2hk\pi}{n}\mathbf{i}}, \mbox{ where $0\leq h\leq n-1$.}\end{equation} Particularly, $\phi_0(a^k)=1$. By Schur's lemma, we have \begin{equation}\label{phi-2}\frac{1}{n}\sum_{k=0}^{n-1}\phi_h(a^k)=\langle \phi_h,\phi_0\rangle=0 \mbox{ for $1\leq h\leq n-1$}.\end{equation}
Theoretically, Theorem \ref{thm-3-1} gives  a necessary and sufficient condition for the integrality of Cayley graphs over dihedral groups. As an application of Theorem \ref{thm-3-1}  we give a class of integral $X(D_n,S)$.
\begin{cor}\label{exa-1}
For odd number $m$, let $D_{2m}=\langle a,b\mid a^{2m}=b^2=1,bab=a^{-1}\rangle$ be the dihedral group of order $4m$. Let $S_1=\{a^m\}, S_2= b\langle a^2\rangle$ and $S=S_1\cup S_2$, then $X(D_{2m},S)$ is connected and integral.
\end{cor}
\begin{proof}
It is easy to see that $S=S^{-1}$ generates $D_n$, and so $X(D_{2m},S)$ is connected.
By Lemma \ref{lem-2-4}, we have $\chi_{_h}(S_1)=2\cos(\frac{2hm\pi}{2m})=\pm2$, and $\chi_{_h}(S_1^2)=2$ due to $S_1^2=\{1\}$. By simple calculation, $S_2^2$ consists of all elements of $\langle a^2\rangle$ in which each one appears $m$ times, that is, $S_2^2=m*\langle a^2\rangle$. From (\ref{phi-1}), (\ref{phi-2}) and Lemma \ref{lem-2-4}, we have
$$
\begin{array}{ll}
\chi_{_h}(S_2^2)&=m\sum_{k=0}^{m-1}\chi_{_h}(a^{2k})=m\sum_{k=0}^{m-1}2\cos\frac{2h\cdot2k\pi}{2m}=m\sum_{k=0}^{m-1}2\cos\frac{2hk\pi}{m}\\
&=m\sum_{k=0}^{m-1}(\cos\frac{2hk\pi}{m}+\cos\frac{2h(m-k)\pi}{m})\\
&=m(\sum_{k=0}^{m-1}e^{\frac{2hk\pi}{m}\mathbf{i}}+\sum_{k=0}^{m-1}e^{\frac{2h(m-k)\pi}{m}\mathbf{i}})\\
&=2m\sum_{k=0}^{m-1}e^{\frac{2hk\pi}{m}\mathbf{i}}\\
&=2m\sum_{k=0}^{m-1}\phi_h((a^2)^k)\\
&=2m\cdot m\langle \phi_h,\phi_0\rangle=0.
\end{array}
$$
By Theorem \ref{thm-3-1}, $X(D_{2m},S)$ is integral.
\end{proof}
In the next section, we will simplify the result of Theorem \ref{thm-3-1} and provide infinite classes of integral Cayley graphs over dihedral groups in terms of Boolean algebra on cyclic groups.

\section{ The necessary and sufficient condition for the integrality of $X(D_n,S)$}

Alperin and Peterson \cite{Alperin} give a necessary and sufficient condition for the integrality of Cayley graphs over abelian groups, in which they introduce some notions such as Boolean algebra and atoms for a group. Let $G$ be a finite group, and $\mathcal{F}_G$ the set consisting of all subgroups of $G$. Then the \emph{Boolean algebra} $B(G)$ is the set whose  elements are obtained by arbitrary finite intersections, unions, and
complements of the elements in $\mathcal{F}_G$. The minimal elements of $B(G)$ are called \emph{atoms}. Clearly, distinct atoms are disjoint. Alperin and Peterson show that each element of $B(G)$ is the union of some atoms, and each atom of $B(G)$ has the form $[g]=\{x\mid \langle x\rangle=\langle g\rangle,x\in G\}$, where $g\in G$.

We say that a subset $S\subseteq G$ is \emph{integral} if $\chi(S)=\sum_{s\in S}\chi(s)$ is an integer for every character $\chi$ of $G$. From Lemma \ref{lem-2-3} we know that $S$ must be an integral set if the Cayley graph $X(G,S)$ is integral. The following elegant result gives a simple characterization of integral Cayley graphs over the abelian group $G$  by using integral set and  atoms of $B(G)$.
\begin{thm}[\cite{Alperin}, Theorem 5.1 and Corollary 7.2]\label{lem-3-1}
Let $G$ be an abelian group. Then $S\subseteq G$ is integral iff $S\in B(G)$ iff $S$ is a union of atoms of $B(G)$ iff  $X(G,S)$ is integral.
\end{thm}

However, the result of Lemma \ref{lem-3-1} is not true for non-abelian group, and is not true for dihedral group. Notice that the dihedral group $D_n$ is the semidirect product of cyclic group $C_n$ by $C_2$, i.e., $D_n=C_n\rtimes C_2$. In what follows we characterize integral $X(D_n,S)$ by using the Boolean algebra of the cyclic group $C_n$.

Let $S$ be a subset of $G$. A \emph{multi-set} based on $S$, denoted by $S^m$, is defined by a multiplicity function $m_S: S\rightarrow \mathbb{N}$, where $m_S(s)$ ($s\in S$) counts how many times $s$ appears in the multi-set. We set $m_S(s)=0$ for $s\not\in S$. The multi-set $S^m$ is called \emph{inverse-closed} if $m_S(s)=m_S(s^{-1})$ for each $s\in S$ and \emph{integral} if $\chi(S^m)=\sum_{s\in S}m_S(s)\chi(s)$ is an integer for  each character $\chi$ of $G$. Besides, if no disagreement would be aroused, we might write $\chi(S^m)=\sum_{s\in S^m}\chi(s)$ instead of $\chi(S^m)=\sum_{s\in S}m_S(s)\chi(s)$.

For  $S\in B(G)$, we know that $S$ is the union of some atoms, say $S=[g_1]\cup [g_2]\cup\cdots\cup [g_k]$. Denote by $S^{m_{g_1,g_2,\ldots,g_k}}$ the multi-set with multiplicity function $m_{g_1,g_2,\ldots,g_k}$, where $m_{g_1,g_2,\ldots,g_k}(s)=m_i\in \mathbb{N}$ for each $s\in [g_i]$ and $1\leq i\leq k$; that is, $S^{m_{g_1,g_2,\ldots,g_k}}=m_1*[g_1]\cup m_2*[g_2]\cup\cdots\cup m_k*[g_k]$.  We define $C(G)=\{S^{m_{g_1,g_2,\ldots,g_k}}\mid S=[g_1]\cup [g_2]\cup\cdots\cup [g_k]\in B(G), g_i\in G, k\in \mathbb{N}\}$ to be the collection of all multi-sets such as $S^{m_{g_1,g_2,\ldots,g_k}}$, which is called the \emph{integral cone} over $B(G)$. By Theorem \ref{lem-3-1}, $T$ is an integral set of the abelian group $G$ iff $T\in B(G)$. First we generalize  this result to integral multi-set $T^\mu$ by using the similar method.
\begin{lem}\label{lem-3-2}
Let $G$ be an abelian group, and $T^m$ a multi-subset of $G$. Then $T^m$ is integral if and only if $T^m\in C(G)$, where $C(G)$ is the integral cone over $B(G)$.
\end{lem}
\begin{proof}
From Theorem \ref{lem-3-1}, each atom is an integral subset of $G$, thus $T^m$ is integral if $T^m\in C(G)$ by the definition of $C(G)$. Next we consider the necessity.

Suppose that $G=\{a_1,\ldots,a_n\}$, and $\chi_1,\ldots,\chi_n$ are irreducible characters of $G$. Consider the matrix $M=(m_{i,j})=(\chi_i(a_j))$. From the orthonormality of characters, the rows of $M$ are independent. The entries of $M$ belong to $Q(\zeta_m)$, where $\zeta_m$ is a primitive $m$-th root of unity and $m$ is the exponent of $G$. Consider the conjugate transpose $\bar{M}^T$, it is well known that $M\bar{M}^T=nI_n$ or $M^{-1}=\frac{1}{n}\bar{M}^T$. For a multi-subset $T^m$ of $G$, let $v_{T^m}$ be the vector with $m(a)$ in the location $a\in G$. Then $Mv_{[a]}=(\chi_1([a]),\ldots,\chi_n([1]))^T\in Z^n$ for each atom $[a]\in B(G)$ because $[a]$ is integral. Let $V=\{v\in Q^n\mid Mv\in Q^n\}$. Hence $W=\textrm{span}_Q\{v_{[a]}\mid \textrm{all atoms } [a]\}\subseteq V$. On the other hand, if $v\in V$ then $Mv=w\in Q^n$; hence $v=\frac{1}{n}\bar{M}^Tw$. Thus $v_i=\frac{1}{n}\sum_j\bar{\chi_j(a_i)}w_j\in Q^n$. Now we index the components of $v$ by the elements of $G$, that is $v_i=v(a_i)$. For any $b\in [a]=\{x\mid\langle x \rangle=\langle x \rangle\}$, there exists some $r$, satisfying $(r,|a|)=1$, such that $b=a^r$. Since the Galois group $Q(\zeta_m)$ over $Q$ acts transitively on roots of unity of the same order, there is an automorphism $\sigma$ taking $\chi_j(a)$ to $\chi_j(b)$ for every $j$.
Therefore, $v(a)=\sigma(v(a))=\frac{1}{n}\sum_j\sigma(\bar{\chi_j(a)})w_j=\frac{1}{n}\sum_j\bar{\chi_j(b)}w_j=v(b)$.
It follows that $v\in W$ and so $V\subseteq W$. Thus $W=V$.

Since the set of $v_{[a]}$, $[a]$ is an atom, is a linear independent set; hence it is also a basis for $W$. If $T^m$ is integral then $Mv_{T^m}=(\chi_1(T^m),\ldots,\chi_n(T^m))\in Z^n$ and so $v_{T^m}\in V=W$. Thus it is a linear combination of $v_{[a]}$, say $v_{T^m}=m_1v_{[a_1]}+m_2v_{[a_2]}+\cdots+m_kv_{[a_k]}$. Thus each element $t\in [a_i]$ ($1\le i\le k$) appears $m_i$ times in $T^m$, i.e., $T^m=m_1*[a_1]\cup m_2*[a_2]\cup\cdots\cup m_k*[a_k]$. It follows that $T^m\in C(G)$.
\end{proof}
\begin{remark} In fact, Matt et al. \cite{Matt} noticed this property when they consider the integral Cayley multigraphs.
\end{remark}
\begin{lem}\label{clm-1}
Let $U$ be a multi-set of integers satisfying $U=-U$, and let $n,h$ be two positive integers. Then
$\sum_{u\in U}\cos\frac{2hu\pi}{n}=\sum_{u\in U}e^{\frac{2hu\pi i}{n}}$.
\end{lem}
\begin{proof}
Since $U=-U$, we have $\sum_{u\in U}e^{\frac{2hu\pi i}{n}}=\sum_{u\in U}e^{\frac{2h(-u)\pi i}{n}}$. Therefore,
$$
\begin{array}{ll}
2\sum_{u\in U}e^{\frac{2hu\pi i}{n}}&
=\sum_{u\in U}e^{\frac{2hu\pi i}{n}}+\sum_{u\in U}e^{\frac{2h(-u)\pi i}{n}}\\
&=\sum_{u\in U}(e^{\frac{2hu\pi i}{n}}+e^{\frac{2h(-u)\pi i}{n}})\\
&=\sum_{u\in U}2\cos\frac{2hu\pi}{n}.
\end{array}
$$
This completes the proof.
\end{proof}

\begin{lem}\label{lem-3-3}
Let $D_n=\langle a,b\mid a^n=b^2=1,bab=a^{-1}\rangle$ be the dihedral group, and $T^m$ an inverse-closed multi-subset of $\langle a\rangle$. Let $\chi_h$ and $\phi_h$ be the irreducible characters of $D_n$ and  $\langle a\rangle$, respectively. We have
\vspace{-0.2cm}
\begin{enumerate}[(1)]
\item $\chi_{_h}(T^m)=2\phi_h(T^m)$ for $1\le h\le [\frac{n-1}{2}]$;
\vspace{-0.2cm}
\item $\chi_{_h}(T^m)$ are integers for all $1\le h\le [\frac{n-1}{2}]$ iff $\phi_h(T^m)$ are integers for all $0\le h\le n-1$.
    \vspace{-0.2cm}
\end{enumerate}
\end{lem}
\begin{proof}
Since $T^m$ is inverse-closed, there exists a multi-set $U=-U$ of integers such that $T^m=\{a^u\mid u\in U\}$.
By Lemma \ref{clm-1}, we have
$$
\begin{array}{ll}
\chi_{h}(T^m)&=\sum_{u\in U}\chi_{h}(a^u)\\
&=\sum_{u\in U}2\cos\frac{2hu\pi}{n}\\
&=2\sum_{u\in U}e^{\frac{2hu\pi i}{n}}\\
&=2\sum_{u\in U}\phi_h(a^u)\\
&=2\phi_h(T^m).
\end{array}
$$
Thus (1) holds.

Notice that both $\chi_{h}(T^m)$ and $\phi_h(T^m)$ are algebraic integers, we claim that $\chi_{h}(T^m)$ is an integer iff $\phi_h(T^m)$ is an integer for each $1\le h\le [\frac{n-1}{2}]$.
Next we consider $\phi_{n-h}(T^m)$ for $1\le h\le [\frac{n-1}{2}]$. By Lemma \ref{clm-1}, we have $$\phi_{n-h}(T^m)=\sum_{u\in U}e^{\frac{2(n-h)u\pi i}{n}}=\sum_{u\in U}\cos\frac{2(n-h)u\pi}{n}=\sum_{u\in U}\cos\frac{2hu\pi}{n}=\sum_{u\in U}e^{\frac{2hu\pi i}{n}}=\phi_{h}(T^m).$$
Thus $\phi_{n-h}(T^m)$ is an integer iff $\phi_{h}(T^m)$ is an integer.
Also note that $\phi_0(T^m)$ and $\phi_{\frac{n}{2}}(T^m)$ (for even $n$) are always integers. It follows (2).
\end{proof}
By Lemmas \ref{lem-3-1} and \ref{lem-3-3}, we have the following result.
\begin{thm}\label{cor-3-1}
Let $D_n=\langle a,b\mid a^n=b^2=1,bab=a^{-1}\rangle$ be the dihedral group, and $T^m$ an inverse-closed multi-set with $T\subseteq\langle a\rangle\subseteq D_n$. Then $\chi_{_h}(T^m)$ are integers for all $1\le h\le [\frac{n-1}{2}]$ if and only if $T^{\mu}\in C(\langle a\rangle)$. In particular, $\chi_{_h}(T)$ are integers for all $1\le h\le [\frac{n-1}{2}]$ if and only if $T\in B(\langle a\rangle)$.
\end{thm}
Recall that the atom of $B(G)$ containing $g\in G$ has the form $[g]=\{x\mid \langle x\rangle=\langle g\rangle,x\in G\}$. Thus, for cyclic group $\langle a\rangle$ of order $n$,
the atom of $B(\langle a\rangle)$ containing $a^d\in \langle a\rangle$, where $d|n$, can be presented as
$[a^d]=\{a^l\mid (l,n)=d\}$, where $(l,n)$ stands for the  greatest common divisor of $l$ and $n$.

\begin{lem}\label{lem-3-4}
Let $D_n=\langle a,b\mid a^n=b^2=1,bab=a^{-1}\rangle$ be the dihedral group and $T\subseteq \langle a\rangle$. If $T\in B(\langle a\rangle)$, then $2\chi_{_h}(T^2)=(\chi_{_h}(T))^2$ for $1\le h\le[\frac{n-1}{2}]$.
\end{lem}
\begin{proof}
Without loss of generality, suppose that $T=[a^{d_1}]\cup[a^{d_2}]\cup\cdots\cup[a^{d_k}]\subseteq\langle a\rangle$ with $d_i|n$ for $i=1,\ldots,k$. Let $\Phi_i=\{1\le l_i\le n\mid (l_i,n)=d_i\}$, we have $[a^{d_i}]=\{a^{l_i}\mid (l_i,n)=d_i\}=\{a^{l_i}\mid l_i\in\Phi_i\}$. Thus, by setting $\Phi=\cup_{i=1}^{i=k}\Phi_i$, we have $T=\{a^l\mid l\in \Phi\}$ and $T^2=\{a^{s+t}\mid s,t\in\Phi\}$. By Lemma \ref{clm-1}, we have
$$
\begin{array}{ll}
2\chi_{_h}(T^2)&=2\sum_{s,t\in \Phi}\chi_{_h}(a^{s+t})\\
&=2\sum_{s,t\in \Phi}2\cos\frac{2h(s+t)\pi}{n}\\
&=4\sum_{s,t\in \Phi}e^{\frac{2h(s+t)\pi}{n}\mathbf{i}}\\
&=4(\sum_{s\in\Phi}e^{\frac{2hs\pi}{n}\mathbf{i}})(\sum_{t\in\Phi}e^{\frac{2ht\pi}{n}\mathbf{i}})\\
&=4(\sum_{s\in\Phi}\cos\frac{2hs\pi}{n})(\sum_{t\in\Phi}\cos\frac{2ht\pi}{n})\\
&=(\sum_{s\in\Phi}2\cos\frac{2hs\pi}{n})(\sum_{t\in\Phi}2\cos\frac{2ht\pi}{n})\\
&=(\sum_{s\in\Phi}\chi_{_h}(a^{s}))(\sum_{t\in\Phi}\chi_{_h}(a^{t}))\\
&=(\chi_{_h}(T))^2.
\end{array}
$$
We complete the proof.
\end{proof}
From Theorem \ref{cor-3-1} and Lemma \ref{lem-3-4}, we get the following corollary immediately.
\begin{cor}\label{cor-3-2}
Let $D_n=\langle a,b\mid a^n=b^2=1,bab=a^{-1}\rangle$ be the dihedral group and $T\subseteq \langle a\rangle$. If $T\in B(\langle a\rangle)$, then $\chi_{_h}(T)$, $\chi_{_h}(T^2)$ are integers and $2\chi_{_h}(T^2)$ is a square number for all $1\le h\le [\frac{n-1}{2}]$.
\end{cor}
Using these preparations above, the result of Theorem \ref{thm-3-1} can be   simplified  as the following theorem.
\begin{thm}\label{thm-3-2}
Let $D_n=\langle a,b\mid a^n=b^2=1,bab=a^{-1}\rangle$ be the dihedral group of order $2n$, and let $S=S_1\cup S_2\subseteq D_n\setminus\{1\}$ be such that $S=S^{-1}$, where $S_1\subseteq\langle a\rangle$ and $S_2\subseteq b\langle a\rangle$. Then $X(D_n,S)$ is integral if and only if $S_1\in B(\langle a\rangle)$ and $2\chi_{_h}(S_2^2)$ is a square number for all $1\le h\le [\frac{n-1}{2}]$.
\end{thm}
\begin{proof}
First suppose that $X(D_n,S)$ is integral. By Theorem \ref{thm-3-1}, $\chi_{_h}(S_1)$ is an integer for all $1\le h\le [\frac{n-1}{2}]$. By Theorem \ref{cor-3-1}, $S_1\in B(\langle a\rangle)$. Then, by Lemma \ref{lem-3-4}, we have $2\chi_{_h}(S_1^2)=(\chi_{_h}(S_1))^2$. Therefore,
$$\Delta_h(S)=2(\chi_{_h}(S_1^2)+\chi_{_h}(S_2^2))-(\chi_{_h}(S_1))^2=2\chi_{_h}(S_2^2).$$
Again by Theorem \ref{thm-3-1}, $\Delta_h(S)=2\chi_{_h}(S_2^2)$ is a square number for all $1\le h\le [\frac{n-1}{2}]$.

Conversely, suppose that $S_1\in B(\langle a\rangle)$ and $2\chi_{_h}(S_2^2)$ is a square number for all $1\le h\le [\frac{n-1}{2}]$. By Corollary \ref{cor-3-2}, both $\chi_{_h}(S_1)$ and $\chi_{_h}(S_1^2)$ are integers for each $h$.  Moreover, $\chi_{_h}(S_2^2)$ must be an integer because $2\chi_{_h}(S_2^2)$ is a square number and $\chi_{_h}(S_2^2)$ is an algebraic integer. Therefore, $\chi_{_h}(S_1^2)+\chi_{_h}(S_2^2)$ is an integer. Since $S_1\in B(\langle a\rangle)$ and $2\chi_{_h}(S_2^2)$ is a square number, by Lemma \ref{lem-3-4}, $\Delta_h(S)=2\chi_{_h}(S_2^2)$ is a square number. Thus $X(D_n,S)$ is integral by Theorem \ref{thm-3-1}.
\end{proof}
Theorem \ref{thm-3-2} gives a criterion to find integral $X(D_n,S)$, from which we get infinite classes of integral Cayley graphs over dihedral groups in the following corollary.
\begin{cor}\label{cor-3-3}
Let $D_n=\langle a,b\mid a^n=b^2=1,bab=a^{-1}\rangle$ be the dihedral group of order $2n$, and let $S=S_1\cup S_2\subseteq D_n\setminus\{1\}$ be such that $S=S^{-1}$, where $S_1\subseteq\langle a\rangle$ and $S_2\subseteq b\langle a\rangle$. If $S_1,bS_2\in B(\langle a\rangle)$, then $X(D_n,S)$ is integral.
\end{cor}
\begin{proof}
By Theorem \ref{thm-3-2}, it suffices to show that $2\chi_{_h}(S_2^2)$ is a square number for all $1\le h\le [\frac{n-1}{2}]$. Since $bS_2\in B(\langle a\rangle)$, $bS_2$ and $S_2$ can be written as
$$bS_2=[a^{d_1}]\cup[a^{d_2}]\cup\cdots\cup[a^{d_k}]\textrm{ and }S_2=b[a^{d_1}]\cup b[a^{d_2}]\cup\cdots\cup b[a^{d_k}],$$
for some $d_i|n$ where $i=1,\ldots,k$.
Therefore,
$$(bS_2)^2=\{a^{l_1+l_2}\mid l_1,l_2\in\Phi\}\textrm{ and }S_2^2=\{a^{l_1-l_2}\mid l_1,l_2\in\Phi\},$$
where $\Phi=\{\ell\mid (\ell,n)\in\{d_1,\ldots,d_k\}\}$.
Note that $\Phi=-\Phi$ $(\mathrm{mod}~n)$, we have $S_2^2=(bS_2)^2$.
Since $bS_2\in B(\langle a\rangle)$, by Lemma \ref{lem-3-4},
$$2\chi_{_h}(S_2^2)=2\chi_{_h}((bS_2)^2)=\chi_{_h}^2(bS_2),$$
is a square number.
\end{proof}
Corollary \ref{cor-3-3} gives an explicit condition for the integrality of $X(D_n,S)$. However, this sufficient condition is not necessary. We will give a counterexample in Example \ref{exa-6}. For this purpose, we need to introduce the famous Remanujan sum. Let $s\ge0,n\ge1$ be two integers, the \emph{Remanujan sum} is defined by $c(s,n)=\sum_{(k,n)=1}e^{\frac{2sk\pi }{n}\mathbf{i}}$, which is  known as an integer (see \cite{George} for reference):
\begin{equation}\label{eq-7}
c(s,n)=\frac{\varphi(n)}{\varphi(\frac{n}{(s,n)})}\mu(\frac{n}{(s,n)})
\end{equation}
where $\varphi(\cdot)$ and $\mu(\cdot)$ are the Euler's totient function and  M\"{o}bius function respectively. The \emph{Euler's totient function} $\varphi(n)$ of a positive integer $n$ is defined as the number of positive integers less than and relatively prime to $n$. The \emph{M\"{o}bius function} $\mu(n)$ of a positive integer $n$ is defined by
$$
\mu(n)=\left\{
\begin{array}{lll}
0&\mbox{if $n$ contains some square divisor;}\\
(-1)^k&\mbox{if $n=p_1\cdots p_k$, where $p_i$ is prime;}\\
1&\mbox{if $n=1$.}
\end{array}
\right.
$$
Therefore, it is easy to see that
$c(s,n)=\mu(n)$ if $(s,n)=1$ and $c(s,n)=\varphi(n)$ if $(s,n)=n$.

\begin{exa}\label{exa-6}
\emph{Let $D_{3k}=\langle a,b\mid a^{3k}=b^2=1,bab=a^{-1}\rangle$ be the dihedral group, and $S=S_1\cup S_2$, where $S_1=[a]$ and $S_2=\{b,ba^{k}\}$. Clearly, $X(D_{3k},S)$ is connected. By simple computation, $S_2^2=2*[1]\cup [a^k]$
which is a multi-subset with multiplicity function: $m_{S_2^2}(x)=2$ if $x\in[1]=[a^n]=\{1\}$, and $m_{S_2^2}(x)=1$ if $x\in [a^k]=\{a^{k},a^{2k}\}$.
Therefore, by Lemma \ref{lem-3-2}, $S_2^2\in C(\langle a\rangle)$ is integral. Moreover, by Lemma \ref{lem-3-3} (1), we have $2\chi_{_h}(S_2^2)=4\phi_h(S_2^2)$. Therefore, $$2\chi_{_h}(S_2^2)=4\phi_h(S_2^2)=4(2e^{\frac{2hn\pi}{n}\mathbf{i}}+\sum_{(l,n)=k}e^{\frac{2hl\pi}{n}\mathbf{i}})=4(2+\sum_{(s,3)=1}e^{\frac{2hs\pi}{3}\mathbf{i}})=4(2+c(h,3)).
$$
If $3\nmid h$, then $(h,3)=1$ and so $c(h,3)=\mu(3)=-1$. It means that $2\chi_{_h}=4(2-1)=4$. If $3|h$, then $(h,3)=3$ and so $c(h,3)=\varphi(3)=2$. It means that $2\chi_{_h}=4(2+2)=16$. Thus $2\chi_{_h}(S_2^2)$ is a square number for all $1\le h\le[\frac{3k}{2}]$. Moreover, $S_1=[a]\in B(\langle a\rangle)$. By Theorem \ref{thm-3-2}, $X(D_{3k},S)$ is integral. However, $bS_2=\{1,a^k\}\not\in B(\langle a\rangle)$.}
\end{exa}
Next we present a necessary condition for the integrality of $X(D_n,S)$.
\begin{cor}\label{cor-3-4}
Let $D_n=\langle a,b\mid a^n=b^2=1,bab=a^{-1}\rangle$ be the dihedral group, and let $S=S_1\cup S_2\subseteq D_n\setminus\{1\}$ be such that $S=S^{-1}$, where $S_1\subseteq\langle a\rangle$ and $S_2\subseteq b\langle a\rangle$. If $X(D_n,S)$ is integral, then $S_1\in B(\langle a\rangle)$ and $S_2^2\in C(\langle a\rangle)$.
\end{cor}
\begin{proof}
If $X(D_n,S)$ is integral,  then $S_1\in B(\langle a\rangle)$ and $2\chi_{_h}(S_2^2)$ is a square number for all $1\le h\le [\frac{n-1}{2}]$ by Theorem \ref{thm-3-2}, and so $\chi_{_h}(S_2^2)$ must be a rational number. Thus we claim that $\chi_{_h}(S_2^2)\in \mathbb{Z}$ because $\chi_{_h}(S_2^2)$ is an algebraic integer. By Theorem  \ref{cor-3-1}, we get $S_2^2\in C(\langle a\rangle)$.
\end{proof}

Unfortunately, the necessary condition given in Corollary \ref{cor-3-4} is not sufficient yet. We present a counterexample below.
\begin{exa}\label{exa-4}
\emph{Let $D_7=\langle a,b\mid a^7=b^2=1,bab=a^{-1}\rangle$ be the dihedral group of order $14$ and $S=\{ba,ba^2,ba^4\}$. It is clear that  $X(D_7,S)$ is connected. By simple computation, we have
$$S^2=\{1,1,1,a,a^2,a^3,a^4,a^5,a^6\}=3*[1]\cup [a]\in C(\langle a\rangle).$$
Therefore, by Lemma \ref{lem-3-2}, $S_2^2\in C(\langle a\rangle)$ is integral. Moreover, by Lemma \ref{lem-3-3} (1), we have $2\chi_{_h}(S_2^2)=4\phi_h(S_2^2)$. Therefore $2\chi_{_h}(S_2^2)=4\phi_h(S_2^2)=4(3+\sum_{k=1}^6e^{\frac{2hk\pi}{7}\mathbf{i}})=4(3+c(h,7)).$
If $7\nmid h$, then $(h,7)=1$ and so $c(h,7)=\mu(7)=-1$. It means that $2\chi_{_h}(S_2^2)=4(3-1)=8$, which is not a square number. By Theorem \ref{thm-3-2}, $X(D_7,S)$ is not integral.}
\end{exa}
Although the integral Caylay graphs over dihedral group are completely characterized by Theorem \ref{thm-3-2}, it seems difficult to evidently give all integral $X(D_n,S)$. In the next section we will determine all integral $X(D_n,S)$ for $n$ being a prime.
\section{Integral Cayley graphs over $D_p$}\label{s-4}
Let $D_n=\langle a,b\mid a^n=b^2=1,bab=a^{-1}\rangle$ be the dihedral group of order $2n$, and let $S=S_1\cup S_2\subseteq D_n\setminus\{1\}$ be such that $S=S^{-1}$, where $S_1\subseteq\langle a\rangle$ and $S_2\subseteq b\langle a\rangle$. We have known that $S_1\in B(\langle a\rangle)$ and $S_2^2\in C(\langle a\rangle)$ if $X(D_n,S)$ is integral by Corollary \ref{cor-3-4}. This implies that $S_2^2$ has the form
\begin{equation}\label{square-0}
S_2^2=m_1*[a^{d_1}]\cup m_2*[a^{d_2}]\cup\cdots\cup m_k*[a^{d_k}],
\end{equation}
for some $d_i|n$ where $i=1,\ldots,k$. The multiplicity function $m_{S_2^2}$ of the multi-set $S_2^2$ is given by $m_{S_2^2}(x)=m_i$ for $x\in [a^{d_i}]$. We say that $S_2^2$ is \emph{$k$-integral} if $S_2^2$ has the form of (\ref{square-0}) with $m_i\not=0$ for $1\leq i\leq k$. Clearly, $S_2^2$ always contains $1$, and if $S_2$ contains two distinct elements then $S_2^2$ will contain an element different from $1$. Taking $k=1$ in (\ref{square-0}), then $S_2$ contains only one element and so $S_2^2=\{1\}$, thus $2\chi_{_h}(S_2^2)=4$ is a square number. Then we have the following result for the $1$-integral $S_2^2$.
\begin{lem}\label{lem-4-1}
Let $D_n=\langle a,b\mid a^n=b^2=1,bab=a^{-1}\rangle$ and $S=S_1\cup S_2\subseteq D_n\setminus\{1\}$ such that $S=S^{-1}$, where $S_1\subseteq\langle a\rangle$ and $S_2\subseteq b\langle a\rangle$. If $S_1\in B(\langle a\rangle)$ and $S_2=\{ba^i\}$ for any $0\le i\le n-1$, then $X(D_n,S)$ is integral.
\end{lem}
In what follows, we focus on $2$-integral $S_2^2$, i.e., $S_2^2=m_1*[a^{d_1}]\cup m_2*[a^{d_2}]$. Since $S_2\subseteq b\langle a\rangle$, there exists $U\subset Z_n$ such that $S_2=\{ba^i\mid i\in U\}$. Then $S_2^2=\{a^{u_1-u_2}\mid u_1,u_2\in U\}$ is a multi-set containing $t*\{1\}$, where $t=|U|=|S_2|$. Thus, without loss of generality, we always assume that  $m_1=t$ and $d_1=n$. By Lemma \ref{lem-3-3} (1),  $\chi_{_h}(S_2^2)=2\phi_h(S_2^2)$ for $1\le h\le [\frac{n-1}{2}]$. Thus $2\chi_{_h}(S_2^2)$ is a square number if and only if $\phi_h(S_2^2)$ is a square number. The following result can determine if $\phi_h(S_2^2)$ is a square number for the $2$-integral $S_2^2$.
\begin{lem}\label{lem-4-2}
Let $n=p_1^{\alpha_1}p_2^{\alpha_2}\cdots p_r^{\alpha_r}$ ($p_i\ge3$, $r\ge 1$) be the prime factorization of $n$ and $S_2\subseteq b\langle a\rangle\subseteq D_n=\langle a,b\mid a^n=b^2=1,bab=a^{-1}\rangle$. If $S_2^2=t*[1]+m_2*[a^{d_2}]$ (implying that $|S_2|=t$) is $2$-integral,
then $\phi_h(S_2^2)$ is a square number for all $1\le h\le[\frac{n-1}{2}]$ if and only if $\frac{n}{d_2}=p_i$ for some $i$ ($1\le i\le r$), and $t=p_i-1$ or $p_i$.
\end{lem}
\begin{proof}
Let $n_2=\frac{n}{d_2}$, according to (\ref{phi-1}) and (\ref{eq-7}) we have
\begin{eqnarray}\label{square-1}
\phi_h(S_2^2)&=&\sum_{x\in S_2^2}\phi_h(x)
=t\phi_h(1)+\sum_{x\in [a^{d_2}]}m_2\phi_h(x)\nonumber\\
&=&t+m_2\sum_{(l,n)=d_2}e^{\frac{2hl\pi}{n}\mathbf{i}}=t+m_2\sum_{(l/d_2,n_2)=1}e^{\frac{2hl/d_2\pi}{n_2}\mathbf{i}}\nonumber\\
&=&t+m_2c(h,n_2)=t+m_2\frac{\varphi(n_2)}{\varphi(\frac{n_2}{(h,n_2)})}\mu(\frac{n_2}{(h,n_2)})
\end{eqnarray}
By counting the number of elements of $S_2^2$, we have
\begin{equation}\label{square-3}
t+m_2\varphi(n_2)=t^2.
\end{equation}
From (\ref{square-1}) and (\ref{square-3}) we have
\begin{equation}\label{square-4-1}
\phi_h(S_2^2)=t+\frac{t(t-1)}{\varphi(\frac{n_2}{(h,n_2)})}\mu(\frac{n_2}{(h,n_2)}).
\end{equation}

Now we consider the sufficiency. Suppose that there exists some $i$ such that $n_2=p_i$, and $t=p_i-1$ or $p_i$. For any $h$ satisfying $(h,n_2)=1$, from (\ref{square-4-1}) we have
\begin{equation}\label{square-4}
\phi_h(S_2^2)=\frac{t(p_i-t)}{p_i-1}.
\end{equation}
Thus $\phi_h(S_2^2)=1$ if $t=p_i-1$, and $\phi_h(S_2^2)=0$  if $t=p_i$. Additionally, for any $h$ satisfying $(h,n_2)=n_2$, from (\ref{square-4-1}) we have $\phi_h(S_2^2)=t^2$. Thus $\phi_h(S_2^2)$ is always a square numbers for  $1\le h\le[\frac{n-1}{2}]$.

Conversely, assume that $\phi_h(S_2^2)$ is a square number, say $\phi_h(S_2^2)=w_h^2$ ($w_h\geq 0$) for $1\le h\le[\frac{n-1}{2}]$. Notice that $n_2$ is a factor of $n$. We need to consider the following three cases:

\textbf{ Case 1.} $n_2=p_i$ for some $i$ ($1\le i\le r$);

Taking $h$ such that $(h,n_2)=1$, then we also have (\ref{square-4}), from which we obtain that $t>w_h^2$ due to $t\ge 2$, and
\begin{equation}\label{square-5}
p_i=\frac{(t+w_h)(t-w_h)}{t-w_h^2}.
\end{equation}
Note that $t-w_h\ge t-w_h^2\ge1$. First assume that $t-w_h=1$. Then $t-w_h=t-w_h^2=1$, and so $w_h=0$ or $w_h=1$. If $w_h=0$, then $t=w_h+1=1$, which is impossible. If $w_h=1$, then $t=2$, and $p_i=3$ by (\ref{square-5}). Next assume that $t-w_h>1$. Since $p_i$ is a prime, from (\ref{square-5}) we have $t-w_h=t-w_h^2$ or $t+w_h=t-w_h^2$. If $t+w_h=t-w_h^2$, then $w_h=0$; if $t-w_h=t-w_h^2$, then $w_h=0$ or $w_h=1$.
Thus, from (\ref{square-5}) we have  $t=p_i$ if $w_h=0$, and $t=p_i-1$ if $w_h=1$.

\textbf{ Case 2.} $p_i^2|n_2$ for some $i$ ($1\le i\le r$);

By taking $h$ such that $\frac{n_2}{(h,n_2)}=n_2$, we have $\mu(\frac{n_2}{(h,n_2)})=\mu(n_2)=0$, which leads to that $\phi_h(S_2^2)=t$ by (\ref{square-4-1}). Thus $t=|S_2|$ is a square number independent with $h$. By taking another $h$ such that $(\frac{n_2}{(h,n_2)})=p_i$, we have $\phi_h(S_2^2)=\frac{t(p_i-t)}{p_i-1}$ by (\ref{square-4-1}), which gives $t=p_i-1$ or $p_i$ by the arguments of Case 1. Finally, if $t=p_i$ then  $p_i$ is a square number, a contradiction; if $t=p_i-1$, we have $t^2=(p_i-1)^2<p_i(p_i-1)\le\varphi(n_2)$, which  contradicts (\ref{square-3}). Thus, in this case, $\phi_h(T^2)$ can not be always a square number for $1\le h\le[\frac{n-1}{2}]$.

\textbf{ Case 3.} $p_ip_j|n_2$ for some $i\ne j$ ($1\le i,j\le r$).

By taking $h$ such that $\frac{n_2}{(h,n_2)}=p_i$, we have $\phi_h(S_2^2)=\frac{t(p_i-t)}{p_i-1}$ by (\ref{square-4-1}), which gives $t=p_i-1$ or $p_i$ by the arguments of Case 1. Similarly, by taking $h$ such that $\frac{n_2}{(h,n_2)}=p_j$, we have $t=p_j-1$ or $p_j$. Note that $t=|S_2|$ is independent with $h$ and $i\ne j$. We have $p_i-1=p_j$ or $p_i=p_j-1$, which are all impossible because both of $p_i$ and $p_j$ are odd primes. Thus $\phi_h(T^2)$ can not be always a square number for $1\le h\le[\frac{n-1}{2}]$ in this case.

We complete the proof.
\end{proof}
By Lemma \ref{lem-4-2}, we  gives a  specific characterization of integral $X(D_n,S)$ for $S_2^2$ being $2$-integral.
\begin{thm}\label{thm-4-1}
Let $n=p_1^{\alpha_1}p_2^{\alpha_2}\cdots p_r^{\alpha_r}$
where $p_i\ge3$ is a prime. Let $1\not\in S=S_1\cup S_2\subseteq D_n=\langle a,b\mid a^n=b^2=1,bab=a^{-1}\rangle$ with  $S=S^{-1}$, $S_1\subseteq\langle a\rangle$ and $S_2\subseteq b\langle a\rangle$. If  $S_2^2$ is a  $2$-integral, then the Cayley graph $X(D_n,S)$ is integral if and only if
$S_1\in B(\langle a\rangle)$ and $S_2=b a^j\langle a^{\frac{n}{p_i}}\rangle \setminus\{ba^{k\frac{n}{p_i}+j}\}$ or
$b a^j\langle a^{\frac{n}{p_i}}\rangle$,
where $1\le i\le r$, $0\le k\le p_i-1$ and $0\le j\le\frac{n}{p_i}-1$.
\end{thm}
\begin{proof}
If $S_2=b a^j\langle a^{\frac{n}{p_i}}\rangle\setminus\{ba^{k\frac{n}{p_i}+j}\}$, we have $$\begin{array}{ll}S_2^2&=(b a^j\langle a^{\frac{n}{p_i}}\rangle\cdot b a^j\langle a^{\frac{n}{p_i}}\rangle \setminus 2*(b a^j\langle a^{\frac{n}{p_i}}\rangle\cdot\{ba^{k\frac{n}{p_i}+j}\})\cup\{ba^{k\frac{n}{p_i}+j}\cdot ba^{k\frac{n}{p_i}+j}\} \\
&=(p_i*\langle a^{\frac{n}{p_i}}\rangle \setminus 2*\langle a^{\frac{n}{p_i}}\rangle)\cup\{1\} \\
&=(p_i-1)*\{1\}\cup (p_i-2)*[a^{\frac{n}{p_i}}].
\end{array}
$$ Similarly,
if $S_2=b\langle a^{\frac{n}{p_i}}\rangle a^j$,
we have $S_2^2=p_i*\{1\}\cup p_i*[a^{\frac{n}{p_i}}]$. Thus, by Lemma \ref{lem-4-2},
$\phi_h(S_2^2)$ is a square number for all $1\le h\le [\frac{n-1}{2}]$,
so are $2\chi_h(S_2^2)$($=4\phi_h(S_2^2)$). By Theorem \ref{thm-3-2}, $X(D_n,S)$ is integral and
the sufficiency follows.

For the necessity, let $S_2^2$ be $2$-integral, we may assume that
\begin{equation}\label{T-eq-1}
S_2^2=t*[1]\cup m_2*[a^{d_2}]
\end{equation}
where $t=|S_2|\ge 2$ and $d_2|n$. Since $X(D_n,S)$ is integral, $\phi_h(S_2^2)$ presented in (\ref{square-4-1}) is a square number for all $1\le h\le [\frac{n-1}{2}]$.
By Lemma \ref{lem-4-2}, there exists $1\le i\le r$ such that $n_2=\frac{n}{d_2}=p_i$ and $t=p_i-1$ or $p_i$.
If $t=p_i-1$,
we have $m_2=\frac{t^2-t}{\varphi(n_2)}=p_i-2$ from (\ref{square-3}). Thus we may assume that $S_2=\{ba^{u_1},\ldots,ba^{u_{p_i-1}}\}$,
combining  (\ref{T-eq-1}) we have
$$(p_i-1)*\{1\}\cup (p_i-2)*[a^{\frac{n}{p_i}}]=S_2^2=\{a^{u_s-u_t}\mid 1\le s,t\le p_i-1\}=(p_i-1)*\{1\}\cup\{a^{u_s-u_t}\mid s\ne t\}.$$
Therefore, $a^{u_s-u_t}\in [a^{\frac{n}{p_i}}]$ for any $s\ne t$.
Since $a^{u_2-u_1},\ldots,a^{u_{p_i-1}-u_1}$ are different elements in $[a^{\frac{n}{p_i}}]$,
there exists $k_1$ ($1\le k_1\le p_i-1$) such that \begin{equation}\label{T-eq-2}
\{a^{u_2-u_1},\ldots,a^{u_{p_i-1}-u_1}\}=[a^{\frac{n}{p_i}}]\setminus\{a^{k_1\frac{n}{p_i}}\}.
\end{equation}
 Note that $u_1$ can be written as $u_1=k_2\frac{n}{p_i}+j$, where $0\le k_2\le p_i-1$ and $0\le j\le\frac{n}{p_i}-1$. Then $a^{u_1}=a^{k_2\frac{n}{p_i}+j}$ and from (\ref{T-eq-2}) we have
\begin{eqnarray*}
bS_2&=&\{a^{u_1},a^{u_2},\ldots,a^{u_{p_i-1}}\}=\{a^{u_1}\}\cup\{a^{u_2-u_1},\ldots,a^{u_{p_i-1}-u_1}\}\cdot \{a^{u_1}\}\\
&=&\{a^{u_1}\}\cup([a^{\frac{n}{p_i}}]\setminus\{a^{k_1\frac{n}{p_i}}\})\cdot \{a^{u_1}\}\\
&=&\{a^{k_2\frac{n}{p_i}+j},a^{(k_2+1)\frac{n}{p_i}+j},\ldots,a^{(k_2+p_i-1)\frac{n}{p_i}+j}\}\setminus\{a^{(k_1+k_2)\frac{n}{p_i}+j}\}\\
&=&\{a^j,a^{\frac{n}{p_i}+j},\ldots,a^{(p_i-1)\frac{n}{p_i}+j}\}\setminus\{a^{k\frac{n}{p_i}+j}\}\ \mbox{ where $0\le k\le p_i-1$}\\
&=&\langle a^{\frac{n}{p_i}}\rangle a^j\setminus\{a^{k\frac{n}{p_i}+j}\}.
\end{eqnarray*}
Thus $S_2=b a^j\langle a^{\frac{n}{p_i}}\rangle\setminus\{ba^{k\frac{n}{p_i}+j}\}$.
Similarly, if $t=p_i$, we have $m_2=p_i$. Assume that $S_2=\{ba^{u_1},\ldots,ba^{u_{p_i}}\}$, combining (\ref{T-eq-1}) we have
 $$p_i*\{1\}\cup p_i*[a^{\frac{n}{p_i}}]=S_2^2=\{a^{u_s-u_t}\mid 1\le s,t\le p_i-1\}=p_i*\{1\}\cup\{a^{u_s-u_t}\mid s\ne t\}.$$
 So we have $\{a^{u_2-u_1},\ldots,a^{u_{p_i}-u_1}\}=[a^{\frac{n}{p_i}}]$.
 Similarly, $a^{u_1}$ can be written as $a^{u_1}=a^{k_2\frac{n}{p_i}+j}$, then we have $bS_2=a^j\langle a^{\frac{n}{p_i}} \rangle$ and so $S_2=b a^j\langle a^{\frac{n}{p_i}}\rangle$.

We complete this proof.
\end{proof}

\begin{thm}\label{cor-4-1}
For an odd prime $p$, let $D_p=\langle a,b\mid a^p=b^2=1,bab=a^{-1}\rangle$ and   $S=S_1\cup S_2$ such that $S=S^{-1}$, where $S_1\subseteq\langle a\rangle$, $S_2\subseteq b\langle a\rangle$. Then the Cayley graph $X(D_p,S)$ is integral if and only if $S_1\in B(\langle a\rangle)$ and $S_2=b\langle a\rangle\backslash ba^k$, $ b\langle a\rangle$ or $\{ba^k\}$ where $0\le k\le p-1$.
\end{thm}
\begin{proof}
If $S_2=\{ba^k\}$, then $X(D_p,S)$ is integral by Lemma \ref{lem-4-1}. If $S_2=b\langle a\rangle\backslash ba^k$, then $S_2^2=(p-1)*\{1\}\cup (p-2)*[a]$, thus $d_2=1$ and $t=|S_2|=p-1$. By Lemma \ref{lem-4-2}, $\phi_h(S_2^2)$ is a square number for all $1\le h\le[\frac{p-1}{2}]$, so are $2\chi_h(S_2^2)$. Therefore, by Theorem \ref{thm-3-2}, $X(D_p,S)$ is integral. Similarly, if $S_2=b\langle a\rangle$, we have $S_2^2=p*\{1\}\cup p*[a]$, and so $X(D_p,S)$ is integral. We get the sufficiency. In what follows we consider the necessity.

Suppose $X(D_p,S)$ is integral, by Corollary \ref{cor-3-4}, we have $S_1\in B(\langle a\rangle)$ and $S_2^2\in C(\langle a\rangle)$. First suppose that $t=|S_2|=1$. We have $S_2=\{ba^k\}$ for some $0\le k\le p-1$. Next suppose that $t=|S_2|>1$. We see that
 $S_2^2$ is $2$-integral because $B(\langle a\rangle)$ has only two atoms. By Theorem \ref{thm-4-1}, we have $|S_2|=p-1$ or $p$. If $|S_2|=p-1$, we obtain that $S_2=b\langle a\rangle\backslash ba^k$ where $0\le k\le p-1$ (note that $S_2=b[a]$ while $k=0$); if $|S_2|=p$, we obtain that  $T=b\langle a\rangle=[b]\cup b[a]$.
\end{proof}
By Corollary \ref{cor-3-4}, if $X(D_n,S)$ is integral then $S_1\in B(\langle a\rangle)$ and $S_2^2\in C(\langle a\rangle)$, which means that $S_1$ and $S_2^2$ are clearly found. If we can obtain $S_2$ from (\ref{square-0}), the integral $X(D_n,S)$ will be finally determined by verifying if $2\chi_{_h}(S_2^2)$ is square number (see Theorem \ref{thm-3-2}). However, it seems difficult to do this even if $n$ is a prime. Example \ref{exa-4} provides of an instance that $S_2^2=\{1,1,1,a,a^2,a^3,a^4,a^5,a^6\}=3*[1]\cup [a]\in C(\langle a\rangle)$, but $S_2=\{ba,ba^2,ba^4\}$ is not of the forms stated in Theorem \ref{cor-4-1}( i.e.,  $S_2=b\langle a\rangle\backslash ba^k$, $ b\langle a\rangle$ or $\{ba^k\}$). Hence $X(D_7,S_2)$ is not integral by Theorem \ref{cor-4-1}. In fact, we have many such instances, say, $S_2=\{ba,ba^3,ba^4,ba^8\}\subset b\langle a\rangle$, where $|\langle a\rangle|=13$, is not of the forms stated in Theorem \ref{cor-4-1}, but $S_2^2=13*[a^{13}]\cup [a]$. Also, $X(D_{13},S_2)$ is not integral by Theorem \ref{cor-4-1}.

\end{document}